\documentclass[10pt]{article}

\usepackage{definitions}

\usepackage{amsmath,amssymb,mathrsfs,amsthm}
\usepackage{bm}
\usepackage[ruled,linesnumbered]{algorithm2e}
\usepackage[numbers,square,sort&compress]{natbib}

\newtheorem{theorem}{Theorem}

\newtheorem{example}[theorem]{Example}
\newtheorem{corollary}[theorem]{Corollary}
\newtheorem{lemma}[theorem]{Lemma}

\usepackage{xcolor}

\SetCommentSty{mycommfont}

\title{A Curved-Path Refinement of the Kalai--Kleitman Diameter Bound}
\author{
    Tianhao Liu\thanks{
        Antai College of Economics and Management, Shanghai Jiao Tong University. 
        \{tianhao.liu, ddge, yinyuye\}@sjtu.edu.cn
    }
    \and Dongdong Ge\footnotemark[1]
    \and Yinyu Ye\footnotemark[1]
}
\date{}

\begin{document}
\maketitle

\begin{abstract}
    Let $\Delta_u(d,n)$ be the maximum graph diameter of a pointed $d$-dimensional polyhedron with $n$ facets. Using a curved-path encoding of the iterated Kalai--Kleitman recurrence, we prove, uniformly over $n\geq d\geq 4$,
    \[
        \Delta_u(d,n) \leq (n-d)^{\log_2 G_d},\qquad
        G_d = (4\ln 2+o(1))\frac{d}{(\ln d)^2},
    \]
    where the asymptotic expression for $G_d$ is understood as $d\to\infty$, improving the exponent of the previous best quasi-polynomial bound by an additional logarithmic factor. With the quantitative $d$-step reduction, we also obtain the complementary excess-based bound
    \[
        \Delta_u(d,n) \leq (n-d)^{\frac{1}{2}\log_2(n-d)+O(1)},
    \]
    where the implied constant is absolute. In the regime $n - d = \Theta(d)$, the latter bound is asymptotically stronger and halves the leading coefficient in the exponent. 
    We further examine their behavior as $n$ grows relative to $d$, obtaining sharper exponents when $n = d^{1/\gamma+o(1)}$ for fixed $0<\gamma<1$ and an almost-linear bound in the deep-tail regime $(\ln n)/d\to\infty$. We also show that the leading term of the general bound is sharp within this positive path-counting framework.
\end{abstract}


\section{Introduction} \label{sec:intro}

Bounding the graph diameter of polyhedra is a fundamental open problem at the intersection of polyhedral geometry, combinatorics, and optimization. The diameter of a pointed polyhedron is the largest, over all pairs of vertices, of the minimum number of edges needed to travel from one vertex to the other. From the perspective of linear programming, a polynomial diameter bound is a necessary geometric prerequisite for a strongly polynomial simplex method. We denote by $\Delta_u(d,n)$ the maximum diameter of a pointed $d$-dimensional polyhedron with $n$ facets, and by $\Delta_b(d,n)$ the corresponding maximum over polytopes.

In 1957, Warren M. Hirsch conjectured, in a question posed to George B. Dantzig, that
\begin{equation} \label{eq:hirsch-polytope}
    \Delta_b(d, n) \leq n - d.
\end{equation}
This is known as the Hirsch conjecture. The Hirsch conjecture has appeared in several closely related forms. A natural extension from polytopes to pointed polyhedra asks whether the same bound continues to hold, namely,
\begin{equation} \label{eq:hirsch-polyhedron}
    \Delta_u(d,n)\leq n - d.
\end{equation}
The pointed-polyhedron version was disproved by Klee and Walkup in 1967~\cite{klee1967d}. Santos later disproved the polytope version by constructing a $43$-dimensional polytope with $86$ facets and diameter at least $44$~\cite{santos2012counterexample}.

On the positive side, the Hirsch bound holds for pointed polyhedra of dimension at most three \cite{klee1965paths,klee1966paths} and for polytopes whenever $n-d\leq 6$ \cite{klee1967d,bremner2011edge}. Furthermore, all known counterexamples only slightly violate the Hirsch bound. This leaves the polynomial Hirsch conjecture open: whether
$\Delta_u(d,n)$ admits a polynomial upper bound in $d$ and $n$.

\paragraph{Notation.} We denote by $\Delta_u(d,n)$ the maximum diameter of a pointed $d$-dimensional polyhedron with $n$ facets, and by $\Delta_b(d,n)$ the corresponding maximum over polytopes. We adopt the convention $\Delta_u(0, 0) = 0$; also, $\Delta_u(d, d) = 0$.
Let $\binom{n}{k}$ denote the number of $k$-element subsets of an $n$-element set. For convenience, write $K = \left\lfloor \log_2\frac{n}{d} \right\rfloor$, $D = \ln d$, and $a = \ln 2$. Let $T_\sigma = \sum_{L\geq 0} \sigma^L 2^{-L(L-1)/2}$ where $\sigma \in \mathbb{R}_+$. 


\subsection{Main Results} \label{sec:main-res}
In this paper, we show improved upper bounds on $\Delta_u(d, n)$ towards the polynomial Hirsch conjecture, although they are still quasi-polynomial rather than polynomial.

In Section~\ref{sec:diam-u}, we establish Theorem~\ref{thm:diam-u}, which gives an improved diameter bound for pointed polyhedra.
\begin{theorem} \label{thm:diam-u}
    There exists a sequence $G_d$, depending only on $d$, such that, for every $n \geq d \geq 4$,
    \begin{equation} \label{eq:diam-u}
        \Delta_u(d,n) \leq (n-d)^{\log_2 G_d},\qquad
        G_d = (4\ln 2+o(1))\frac{d}{(\ln d)^2}.
    \end{equation}
    Here the asymptotic term involving $G_d$ is understood as $d\to\infty$.
\end{theorem}

In Section~\ref{sec:diam-d}, we use the quantitative $d$-step reduction to obtain the excess-based bound in Theorem~\ref{thm:diam-d}.
Let
\begin{equation} \label{eq:T}
    T_\sigma = \sum_{L\geq 0} \sigma^L 2^{-L(L-1)/2}.
\end{equation}
\begin{theorem} \label{thm:diam-d}
    For all $n - d \geq 4$,
    \begin{equation} \label{eq:delta-path-bound-d}
        \Delta_u(d, n) \leq 2 (n-d) T_{n-d}.
    \end{equation}
    Consequently, uniformly in $d$, as $n-d\to\infty$,
    \begin{equation} \label{eq:diam-d}
        \Delta_u(d, n) \leq (n-d)^{\frac{1}{2}\log_2(n-d) + \frac{3}{2} + O\left(\frac{1}{\ln(n-d)}\right)},
    \end{equation}
    where the constant implicit in $O(\cdot)$ is absolute.
\end{theorem}
In Section~\ref{sec:diam-d-order}, we further refine the additive constant $3/2$ in the exponent of \eqref{eq:diam-d}.

Beyond this lower-order improvement, we seek a broader understanding of how the general curved-path bounds can be sharpened as the number of facets increases relative to the dimension. To capture this overall asymptotic progression, Section~\ref{sec:diam-d-dim} considers the moderately growing regime $n=d^{1/\gamma+o(1)}$ for fixed $0<\gamma<1$, while Section~\ref{sec:diam-tail} turns to the deep-tail regime $(\ln n)/d\to\infty$. Together with the two general bounds above, these regimes trace the progression from improved quasi-polynomial exponents to almost-linear behavior.

The resulting bounds are compared in Table~\ref{tab:evo-bounds}, with the leading coefficients of the quasi-polynomial bounds normalized relative to $\log_2 d$. In particular, on fixed polynomial scales, the normalized leading coefficient $1-\gamma/2$ interpolates between $1/2$ and the general value $1$ as $\gamma$ decreases from $1$ to $0$; in the deep-tail regime, the bound becomes almost linear.

\begin{table}[!ht]
    \centering
    \begin{tabular}{l|l|l}
        \toprule
         Regime & Upper bound & Leading behavior \\
        \midrule
         General $n\geq d$ & $(n-d)^{\log_2 \left( (4\ln 2 + o(1))d / (\ln d)^2 \right)}$ & coefficient $1$, with a $2\log_2 \ln d$ saving \\
         $n - d = \Theta(d)$ & $(n-d)^{\left( 1/2 + o(1) \right)\log_2 d}$ & coefficient $1/2$ \\
         $n = d^{1/\gamma + o(1)}$, $0<\gamma<1$ fixed & $(n-d)^{\left( 1 - \gamma/2 + o(1) \right)\log_2 d}$ & coefficient $1 - \gamma/2$ \\
         $(\ln n)/d\to\infty$ & $n^{1 + o(1)}$ & almost linear \\
        \bottomrule
    \end{tabular}
    \caption{Curved-path diameter bounds across the principal asymptotic regimes.}
    \label{tab:evo-bounds}
\end{table}


All of these results are based on estimates for the number of admissible curved paths, which are formally introduced in Section~\ref{sec:path-encoding}. In Section~\ref{sec:sharp-u}, we investigate the sharpness of the path-counting argument underlying the general upper bound \eqref{eq:diam-u} and provide evidence that the leading term in its exponent cannot be improved merely by refining the same positive path-counting scheme.

\subsection{Related Work} \label{sec:rel-work}
In 1992, Kalai and Kleitman established the first quasi-polynomial upper bound, $\Delta_u(d,n)\leq n^{2+\log_2 d}$, through an induction based on what is now known as the Kalai--Kleitman inequality~\cite{kalai1992quasi}. Kalai subsequently sharpened this estimate in 1997 to $\Delta_u(d,n)\leq n^{1+\log_2 d}$~\cite{kalai1997linear}. Todd further improved it in 2014 to $\Delta_u(d,n)\leq(n-d)^{\log_2 d}$~\cite{todd2014improved}.

Prior to the present work, the best known general bound was due to Sukegawa in 2019~\cite{sukegawa2019asymptotically}:
\[
    \Delta_u(d,n) \leq (n-d)^{\log_2 \left(\frac{d}{\ln d}\right) + O(1)}.
\]
Our result achieves an additional logarithmic saving in the exponent:
\[
    \Delta_u(d,n)\leq(n-d)^{\log_2 G_d},
    \qquad G_d = (4\ln 2+o(1))\frac{d}{(\ln d)^2}.
\]
Moreover, the quantitative $d$-step reduction yields the excess-based bound $\Delta_u(d,n)\leq(n-d)^{\frac{1}{2}\log_2(n-d)+O(1)}$. When $n-d=\Theta(d)$, this halves the leading coefficient relative to the previously known dimension-based quasi-polynomial exponent.

Alongside the quasi-polynomial bounds discussed above, there are bounds that are linear in the number of facets but exponential in the dimension. Larman proved that $\Delta_b(d,n)\leq 2^{d-3}n$ for polytopes with $n\geq d\geq3$~\cite{larman1970paths}. Through results on normal simplicial complexes, Labb\'e et al. extended the corresponding estimate to pointed polyhedra~\cite{labbe2017hirsch}. Such bounds are particularly effective when $n$ is very large relative to $d$, whereas quasi-polynomial bounds are generally more competitive when $n$ is comparable to, or only moderately larger than, $d$.

Bounds designed specifically for the regime in which the number of facets is large relative to the dimension are commonly referred to as tail bounds. Although the polynomial Hirsch conjecture remains open in full generality, polynomial and even almost-linear estimates are known sufficiently far into the tail~\cite{gallagher2016tail,mizuno2016simple}.

For completeness, we show that the curved-path bound developed here also recovers an almost-linear estimate for pointed polyhedra. This consequence is not intended to improve the best known tail threshold: indeed, the generalized Larman bound already implies $\Delta_u(d,n)\leq n^{1+\epsilon}$ whenever $n\geq 2^{(d-3)/\epsilon}$. Rather, its role is to demonstrate that the same curved-path estimate also recovers the expected almost-linear behavior in the deep-tail regime.



\section{Preliminaries} \label{sec:pre}

Our results stay within the framework of the Kalai--Kleitman recurrence, which is based on the classical inequality in Lemma~\ref{lem:KK}.
\begin{lemma}[Kalai--Kleitman inequality~\cite{kalai1992quasi}] \label{lem:KK}
    For $\lfloor n/2 \rfloor \geq d \geq 2$,
    \begin{equation} \label{eq:KK}
        \Delta_u(d, n) \leq \Delta_u(d-1, n-1) + 2\Delta_u(d, \lfloor n/2 \rfloor) + 2.
    \end{equation}
\end{lemma}

Using Lemma~\ref{lem:KK} and induction, Gallagher and Kim derived the following inequality~\cite{gallagher2016tail}.
\begin{lemma}[Gallagher--Kim {\cite[Lemma~6.1]{gallagher2016tail}}] \label{lem:GK}
    For $n\geq d\geq 3$,
    \begin{equation} \label{eq:GK}
        \Delta_u(d, n) \leq \sum_{i=0}^{\lfloor \log_2(n/d) \rfloor} 2^i \Delta_u(d-1, \lfloor n/2^i \rfloor) - 1.
    \end{equation}
\end{lemma}

Our argument recursively expands \eqref{eq:GK} until only dimension-three terms remain, where the Hirsch bound holds for both polytopes and pointed polyhedra; see Theorems~\ref{thm:diam-3b} and~\ref{thm:diam-3u}.

\begin{theorem}[Klee \cite{klee1966paths}] \label{thm:diam-3b}
    \[
        \Delta_b(3, n) = \left\lfloor \frac{2n}{3} \right\rfloor - 1.
    \]
\end{theorem}

\begin{theorem}[Kim--Santos {\cite[Proposition~3.10]{kim2010update}}] \label{thm:diam-3u}
    \[
        \Delta_u(3, n) = n - 3.
    \]
\end{theorem}


The following theorem shows that the diameter can be bounded in terms of a new dimension and number of facets determined entirely by the excess $n-d$.
\begin{theorem}[Klee--Kleinschmidt {\cite[Statement~2.5]{klee1987d}}] \label{thm:d-step}
    For all $n > d$,
    \begin{equation} \label{eq:d-step}
        \begin{aligned}
            \Delta_u(d, n) & \leq \Delta_u(n-d, 2(n-d)) \\
            \Delta_b(d, n) & \leq \Delta_b(n-d, 2(n-d)),
        \end{aligned}
    \end{equation}
    with equalities when $n \leq 2d$.
\end{theorem}
We call \eqref{eq:d-step} a quantitative $d$-step reduction because it bounds the original diameter by one in dimension $n-d$ with $2(n-d)$ facets, precisely the classical $d$-step regime \cite{klee1967d}.


\section{A Curved-Path Refinement of the Diameter Bound} \label{sec:diam}

In this section, we derive asymptotic quasi-polynomial upper bounds for the diameters of pointed polyhedra. Note that each application of \eqref{eq:GK} reduces the dimension by one. We therefore iterate this inequality until the dimension reaches three and then estimate the resulting nested sums.

\subsection{Exact Curved-Path Encoding} \label{sec:path-encoding}

We regard each diameter term of dimension $q-1$ in the iterated expansion as being generated from a diameter term of dimension $q$ by selecting a summation index $i_q$ in \eqref{eq:GK}. Thus, $i_q$ records the dyadic reduction in the number of facets made when the dimension decreases from $q$ to $q-1$.

\begin{example}
    Consider expanding $\Delta_u(5, 10)$ down to dimension three. We have
    \[
        \Delta_u(5, 10) \leq 2^0\Delta_u(4, 10) + 2^1\Delta_u(4, 5) - 1,
    \]
    and
    \[
        \begin{aligned}
            \Delta_u(4, 10) & \leq 2^0\Delta_u(3, 10) + 2^1\Delta_u(3, 5) - 1 \\
            \Delta_u(4, 5) & \leq 2^0\Delta_u(3, 5) - 1.
        \end{aligned}
    \]
    Thus by Lemma~\ref{lem:GK}, we have
    \[
        \begin{aligned}
            \Delta_u(5, 10) & \leq 2^0\Delta_u(4, 10) + 2^1\Delta_u(4, 5) - 1 \\
            & \leq 2^0 \left[ 2^0\Delta_u(3, 10) + 2^1\Delta_u(3, 5) - 1 \right]
            + 2^1 \left[ 2^0\Delta_u(3, 5) - 1 \right]
            - 1 \\
            & = 2^0 \Delta_u(3, 10) + (2^1 + 2^1) \Delta_u(3, 5) - (2^0 + 2^1) - 1 \\
            & \leq 2^0 \Delta_u(3, 10) + (2^1 + 2^1) \Delta_u(3, 5).
        \end{aligned}
    \]
    Here $\Delta_u(3, 10)$ comes from the reduction
    \[
        \Delta_u(5, 10) \to 2^0 \Delta_u(4, 10) \to 2^{0+0} \Delta_u(3, 10),
    \]
    which is encoded by $(i_5,i_4)=(0,0)$.
    
    The two copies of $2^1 \Delta_u(3, 5)$ arise from the two reduction paths
    \[
        \Delta_u(5, 10) \to 2^0 \Delta_u(4, 10) \to 2^{0+1} \Delta_u(3, 5),
    \]
    and
    \[
        \Delta_u(5, 10) \to 2^1 \Delta_u(4, 5) \to 2^{1+0} \Delta_u(3, 5).
    \]
    These paths are encoded by $(i_5, i_4)=(0, 1)$ and $(i_5, i_4)=(1, 0)$, respectively.
\end{example}

The preceding example suggests encoding each recursion branch by recording the dimensions at which halvings occur. More precisely, if the summation index at dimension $q$ takes the value $i_q$, we record the label $q$ exactly $i_q$ times. This produces a finite weakly decreasing sequence, which we call a curved path. We refer to the terminal dimension-three term reached by a recursion branch as its leaf. The following lemma makes the correspondence between curved paths and recursion leaves precise.

\begin{lemma}[Exact curved-path encoding] \label{lem:path-encoding}
    After discarding the negative internal terms and expanding the positive terms in \eqref{eq:GK} down to dimension three, recursion leaves are in bijection with finite weakly decreasing sequences satisfying
    \begin{equation} \label{eq:curved}
        d \geq \kappa_1 \geq \kappa_2 \geq \cdots \geq \kappa_J\geq 4,
        \qquad \kappa_j2^j \leq n,\ \forall j.
    \end{equation}
    If the sequence has length $J$, the corresponding leaf contributes $2^J\Delta_u(3, \lfloor n/2^J \rfloor)$.
\end{lemma}
\begin{proof}

    For $q = d, d-1,\ldots, 4$, let $i_q \geq 0$ denote the summation index selected when \eqref{eq:GK} is applied at dimension $q$, and set $ S_{q} = \sum_{t=q}^d i_t$, $S_{d+1} = 0 $. Before $i_q$ is selected, the current number of facets is $\lfloor n / 2^{S_{q+1}} \rfloor$.

    We call an index vector admissible if every index lies in the corresponding summation range in \eqref{eq:GK}, i.e.,
    \[
        0 \leq i_q \leq \left\lfloor \log_2 \left( \frac{\lfloor n/2^{S_{q+1}} \rfloor}{q} \right) \right\rfloor,\qquad 4 \leq q\leq d.
    \]
    Each admissible index vector determines a unique recursion leaf, and every recursion leaf arises in this way.

    The nested-floor identity
    \[
        \left\lfloor\frac{\lfloor n/2^s\rfloor}{2^i}\right\rfloor
        = \left\lfloor\frac n{2^{s+i}}\right\rfloor
    \]
    shows that the admissibility conditions are equivalent to
    \[
        q 2^{S_q} \leq n,\qquad 4 \leq q \leq d.
    \]
    Given an admissible index vector, form a sequence by writing the label $q$ exactly $i_q$ times, in decreasing order of $q$. The resulting sequence
    \[
        \kappa_1 \geq \kappa_2 \geq \cdots \geq \kappa_J
    \]
    is weakly decreasing and has length $J = S_4 = \sum_{t=4}^d i_t$. Moreover, $\kappa_j 2^j \leq \kappa_j 2^{S_{\kappa_j}} \leq n$ for all $j$. Thus, the sequence satisfies \eqref{eq:curved}.

    Conversely, given a sequence satisfying \eqref{eq:curved}, let $i_q$ be the multiplicity of $q$ in the sequence. This uniquely determines the index vector $(i_d,\ldots,i_4)$. Then $S_q = \sum_{t=q}^d i_t$ is precisely the number of entries in the sequence that are at least $q$. If $S_q > 0$, then $\kappa_{S_q} \geq q$ and $q 2^{S_q} \leq \kappa_{S_q} 2^{S_q} \leq n$. If $S_q = 0$, the same inequality follows from $q \leq d\leq n$. Thus the reconstructed index vector is admissible, completing the bijection.

    Finally, since $\sum_{q=4}^d i_q=J$, the coefficients multiply to $\prod_{q=4}^d2^{i_q} = 2^J$, and the nested-floor identity gives the terminal facet count $\lfloor n/2^J\rfloor$. Thus the corresponding leaf contributes $2^J\Delta_u(3, \lfloor n/2^J \rfloor)$.

    
\end{proof}
    

The word ``curved'' refers to the position-dependent constraint
\[
    \kappa_j2^j \leq n,
\]
which requires the encoded path to lie not merely inside the rectangular region imposed by the uniform dimension-only constraint $\kappa_j\leq d$, but also below the exponentially decreasing boundary $\kappa_j 2^j = n$.

The curved-path counting argument is inspired by the iterated-recurrence analysis of Gallagher and Kim~\cite{gallagher2016tail}, but retains more of the dimension-dependent structure. Their counting step effectively replaces the running dimension label $\kappa\geq4$ by the uniform lower bound $4$, thereby removing its dependence on the recursion depth. We instead keep the exact constraint $\kappa_j2^j\leq n$, so that the admissible paths lie below a decreasing curved boundary. Preserving this coupling between dimension and recursion depth is the source of our improvement.

\subsection{An Improved Uniform Bound for Pointed Polyhedra} \label{sec:diam-u}

In this section, we prove Theorem~\ref{thm:diam-u} by bounding the number of admissible curved paths.

\begin{theorem}[Curved-path bound] \label{thm:curve-bd}
    For every $d \geq 4$ and $n \geq 2d$,
    \begin{equation} \label{eq:delta-path-bound}
        \Delta_u(d, n) \leq n \binom{d + K}{K}(K + T_d),
    \end{equation}
    where $K = \left\lfloor \log_2 \frac{n}{d} \right\rfloor$ and $T_d$ is defined in \eqref{eq:T}.
\end{theorem}
\begin{proof}

    A path with $J$ halvings has terminal weight at most
    \begin{equation} \label{eq:coeff-n}
        2^J\Delta_u(3, \lfloor n/2^J \rfloor) \leq 2^J \lfloor n/2^J \rfloor \leq n,
    \end{equation}
    where the first inequality comes from the fact that the Hirsch conjecture holds at $d=3$.

    For $J < K$, the constraint $\kappa_j 2^j\leq n$ holds automatically, since
    \[
        \kappa_j2^j \leq d 2^K \leq n,\qquad 1 \leq j \leq J.
    \]
    Hence, the number of weakly decreasing sequences of $J$ dimensions equals the number of nonnegative integer solutions to
    \[
        \sum_{q=4}^{d} i_q = J,
    \]
    which is $\binom{d+J-4}{J}$. We have
    \begin{equation} \label{eq:number-paths}
        \binom{d+J-4}{J} = \binom{d+J-4}{d-4} \leq \binom{d+K-4}{d-4} = \binom{d+K-4}{K} \leq \binom{d+K}{K}.
    \end{equation}
    There are $K$ possible values of $J$.

    For $J = K+L$, the first $K$ dimensions have at most the same number of possibilities in \eqref{eq:number-paths}. Since $n \leq d 2^{K+1}$ by the definition of $K$, condition~\eqref{eq:curved} gives
    \[
        \kappa_{K+\ell} < d 2^{1 - \ell},\qquad 1 \leq \ell\leq L.
    \]
    Dropping the ordering among these last $L$ dimensions bounds their number by
    \begin{equation} \label{eq:last-L}
        \prod_{\ell = 1}^{L} d 2^{1 - \ell} = d^L 2^{\sum_{\ell=1}^L(1-\ell)} = d^L 2^{-L(L-1)/2}.
    \end{equation}
    Summing over all possible $J$, by \eqref{eq:coeff-n}, \eqref{eq:number-paths} and \eqref{eq:last-L}, we have
    \[
        \Delta_u(d, n) \leq n \binom{d+K}{K} K + n \binom{d+K}{K} \sum_{L\geq 0}d^L 2^{-L(L-1)/2} = n \binom{d+K}{K} (K + T_d).
    \]
    
\end{proof}

For convenience, write $D = \ln d$ and $a = \ln 2$. Then we bound $\ln T_d$ using Lemma~\ref{lem:lnT}.
\begin{lemma} \label{lem:lnT}
    As $\sigma \to \infty$,
    \begin{equation} \label{eq:lnT}
        \ln T_\sigma = \frac{(\ln \sigma)^2}{2a} + \frac{\ln \sigma}{2} + O(1),
    \end{equation}
    where the constant implicit in $O(1)$ is absolute.
\end{lemma}
\begin{proof}
    We have
    \[
    \begin{aligned}
        T_\sigma & = \sum_{L\geq 0} \sigma^L 2^{-L(L-1)/2} \\
        & = \sum_{L\geq 0} e^{L\ln \sigma - \frac{a}{2}L(L-1)} \\
        & = e^{\frac{(\ln \sigma)^2}{2a} + \frac{\ln \sigma}{2} + \frac{a}{8}} \sum_{L\geq 0} e^{- \frac{a}{2}\left( L - \frac{\ln \sigma}{a} - \frac{1}{2} \right)^2}.
    \end{aligned}
    \]
    Thus,
    \[
        \ln T_\sigma = \frac{(\ln \sigma)^2}{2a} + \frac{\ln \sigma}{2} + \frac{a}{8} + \ln \left( \sum_{L\geq 0} e^{- \frac{a}{2}\left( L - \frac{\ln \sigma}{a} - \frac{1}{2} \right)^2} \right).
    \]

    Let
    \[
        \mu = \frac{\ln \sigma}{a}+\frac{1}{2}.
    \]
    For sufficiently large $\sigma$, we have $\mu > 0$. Hence, there exists an integer $L_0\geq 0$ such that $\left| L_0 - \mu \right| \leq \frac{1}{2}$, and
    \[
        \sum_{L\geq 0} e^{- \frac{a}{2}\left( L - \mu \right)^2} \geq e^{- \frac{a}{2} \left( L_0 - \mu \right)^2} \geq e^{-\frac{a}{8}} > 0.
    \]
    For the upper bound, $e^{- \frac{a}{2}\left( L - \mu \right)^2}$ increases over $L < \mu$ and decreases over $L > \mu$. Moreover, the larger of the two terms $e^{- \frac{a}{2}\left( \lfloor \mu \rfloor - \mu \right)^2}$ and $e^{- \frac{a}{2}\left( \lfloor \mu \rfloor + 1 - \mu \right)^2}$ is at most $1$, while the smaller is bounded by $\int_{\lfloor \mu \rfloor}^{\lfloor \mu \rfloor + 1} e^{- \frac{a}{2}\left( x - \mu \right)^2} \mathrm{d} x$. Thus,
    \[
    \begin{aligned}
        \sum_{L\geq 0} e^{- \frac{a}{2}\left( L - \mu \right)^2} & \leq \sum_{L\in\mathbb{Z}} e^{- \frac{a}{2}\left( L - \mu \right)^2} \\
        & = \sum_{L\leq \lfloor \mu \rfloor - 1} e^{- \frac{a}{2}\left( L - \mu \right)^2} 
        + e^{- \frac{a}{2}\left( \lfloor \mu \rfloor - \mu \right)^2} 
        + e^{- \frac{a}{2}\left( \lfloor \mu \rfloor + 1 - \mu \right)^2} 
        + \sum_{L\geq \lfloor \mu \rfloor + 2} e^{- \frac{a}{2}\left( L - \mu \right)^2} \\
        & \leq \int_{-\infty}^{\lfloor \mu \rfloor} e^{- \frac{a}{2}\left( x - \mu \right)^2} \mathrm{d} x 
        + 1
        + \int_{\lfloor \mu \rfloor}^{\lfloor \mu \rfloor + 1} e^{- \frac{a}{2}\left( x - \mu \right)^2} \mathrm{d} x 
        + \int_{\lfloor \mu \rfloor + 1}^{+\infty} e^{- \frac{a}{2}\left( x - \mu \right)^2} \mathrm{d} x \\
        & = 1
        + \int_{-\infty}^{+\infty} e^{- \frac{a}{2}\left( x - \mu \right)^2} \mathrm{d} x \\
        & = 1 + \sqrt{\frac{2\pi}{a}}.
    \end{aligned}
    \]
    Therefore, $\ln\left(\sum_{L\geq 0} e^{- \frac{a}{2}\left( L - \mu \right)^2}\right)$ is uniformly bounded, and \eqref{eq:lnT} holds.

\end{proof}

\begin{lemma} \label{lem:sup-Fdk}
    Let
    \begin{equation} \label{eq:def-FdK}
        F_d(K) = \frac{a\left[D + a(K+1) + \ln\binom{d+K}{K} + \ln(K+T_d)\right]}{D + \ln(2^K - 1)}.
    \end{equation}
    As $d \to \infty$,
    \begin{equation} \label{eq:sup-Fdk}
        \sup_{K\in\mathbb{Z}_{\geq 1}} F_d(K)
        = D - 2 \ln D + \ln(4a) + o(1).
    \end{equation}
\end{lemma}
\begin{proof}
    
    We divide the optimization into the two ranges $K<d$ and $K\geq d$, and begin with the former.
    
    \paragraph{Range $\bm{K < d}$.}
    For $K < d$, we have, uniformly,
    \begin{equation} \label{eq:binom}
        \begin{aligned}
            \ln \binom{d+K}{K} & = \sum_{j=1}^{K} \ln(d+j) - \ln (K!) \\
            & = KD + \sum_{j=1}^{K} \ln\left(1+\frac{j}{d}\right) - \ln (K!) \\
            & = KD + O\left( \frac{K^2}{d} \right) - K\ln K + K + O\left(\ln (K + 1)\right) \\
            & = K(D - \ln K + 1) + O\left( \ln(K + 1) + \frac{K^2}{d} \right).
        \end{aligned}
    \end{equation}
    The third equality comes from the elementary logarithmic bound $0 \leq \ln(1 + x) \leq x$ for all $x\geq 0$, and Stirling's formula $\ln (K!) = K\ln K - K + O\left(\ln (K + 1)\right)$.
    
    Moreover, from Lemma~\ref{lem:lnT}, $K = o(T_d)$ in this range gives
    \begin{equation} \label{eq:lnKTd}
        \ln (K + T_d) = \ln T_d + \ln \left( 1 + \frac{K}{T_d} \right) = \ln T_d + o(1),
    \end{equation}
    and $D + \ln(2^K - 1) = D + aK + O(1)$. Hence, whenever $K/D\to\infty$ and $K < d$, we have
    \begin{equation} \label{eq:FdK}
        F_d(K) = D - \ln K + 1 + a - \frac{D^2}{2(aK + D)} + O\left( \frac{D\ln K}{K} + \frac{K}{d} + \frac{\ln(K+1)}{K} \right).
    \end{equation}

    Throughout the following analysis, set
    \[
        y = \frac{aK}{D^2},
    \]
    and thus the principal part in \eqref{eq:FdK} becomes
    \[
        D - \ln K + 1 + a - \frac{D^2}{2(aK + D)} = D - 2\ln D + 1 + a + \ln a - \ln y - \frac{1}{2(y + 1/D)}.
    \]
    We now consider $K$ in increasing order of magnitude. Passing to a subsequence whenever necessary, it suffices to consider the following five regimes:
    \[
        K=O(D),\qquad
        D=o(K),\ K=o(D^2),\qquad
        K=\Theta(D^2),\qquad
        D^2=o(K),\ K\leq\frac{d}{D},\qquad
        \frac{d}{D}<K<d.
    \]
    
    \begin{itemize}
        \item Suppose $K = O(D)$. Uniformly on the range $K/D \leq c$ for some absolute constant $0 < c < \infty$, direct division using \eqref{eq:lnT} and \eqref{eq:binom} gives
        \[
            \frac{F_d(K)}{D} = 1 - \frac{1/2}{aK/D + 1} + o(1) \leq 1 - \frac{1/2}{ac + 1} + o(1).
        \]
        Thus $F_d(K)$ has a leading coefficient strictly smaller than one in this regime.
    
        \item Suppose that $K/D \to \infty$ and $K/D^2 \to 0$. Then $y\to 0$, and the error term in \eqref{eq:FdK} is $o(1/y)$. Indeed, multiplying each component of the error term by $y$ gives $o(1)$. Since $1/y$ grows faster than $-\ln y$, it follows that
        \[
            F_d(K)-(D-2\ln D) \to -\infty.
        \]
    
        \item Suppose that $K = \Theta(D^2)$. Then $y$ remains in a compact subinterval of $(0,\infty)$, and \eqref{eq:FdK} becomes
        \[
            F_d(K) = D-2\ln D+1+a+\ln a-\ln y-\frac{1}{2y}+o(1).
        \]
        The $y$-dependent part $-\ln y-\frac{1}{2y}$ has a unique maximum at $y=1/2$. Since $a=\ln2$, the largest value in this regime is
        \[
            D-2\ln D+\ln(4a)+o(1).
        \]
    
        \item Suppose that $K/D^2\to\infty$ and $K\leq d/D$. Then $y\to\infty$. The error term in \eqref{eq:FdK} is $o(1)$, while $-\ln y \to -\infty$. Therefore,
        \[
            F_d(K) - (D - 2\ln D) \to -\infty.
        \]
    
        \item Finally, if $d/D<K<d$, then
        \[
            D - \ln K \leq \ln D.
        \]
        The error term in \eqref{eq:FdK} is $O(1)$, and consequently
        \[
            F_d(K) = O(\ln D).
        \]
    \end{itemize}
    
    Thus, within the range $K<d$, only the scale \(K=\Theta(D^2)\) can attain $F_d(K) = D - 2\ln D + O(1)$.
    
    \paragraph{Range $\bm{K\geq d}$.}
    For $K\geq d$, Stirling's estimate and the monotonicity of the logarithm give, respectively,
    \begin{equation} \label{eq:binom-dK-K-bound}
        \binom{d + K}{K} \leq \frac{(d+K)^d}{(d/e)^d} = \left( e \left( 1 + \frac{K}{d}\right) \right)^d,
    \end{equation}
    and
    \[
        \ln(K+T_d)
        \leq
        \ln2+\max\{\ln K,\ln T_d\}.
    \]
    Substitution into the definition of $F_d(K)$ shows that
    \[
        F_d(K)=O(1)
    \]
    uniformly throughout this range.
    
    The above cases exhaust all sequences $K=K(d)$. Finally, choosing $K$ to be the nearest integer to $D^2/(2a)$ gives $y=1/2+o(1)$ and supplies the matching bound. Therefore,
    \[
        \sup_{K\in\mathbb{Z}_{\geq1}}F_d(K)
        =
        D-2\ln D+\ln(4a)+o(1).
    \]

\end{proof}


We now prove Theorem~\ref{thm:diam-u}.
\begin{proof}[Proof of Theorem~\ref{thm:diam-u}]

    Choose $G_d$ such that
    \[
        \ln G_d \geq \sup_{K\in\mathbb Z_{\geq1}}F_d(K),
        \qquad d \geq 4,
    \]
    with equality for all sufficiently large $d$. The finitely many exceptional values of $G_d$ will be specified at the end of the proof.
    By Lemma~\ref{lem:sup-Fdk},
    \[
        G_d = (4\ln 2+o(1))\frac{d}{(\ln d)^2}.
    \]

    We consider two cases: $n \geq 2d$ and $n < 2d$.
    
    \paragraph{Case $\bm{n \geq 2d}$.}

    Taking logarithms on both sides of \eqref{eq:diam-u}, it remains to show that
    \begin{equation} \label{eq:diam-u-goal}
        \frac{a \ln \Delta_u(d, n)}{\ln (n - d)} \leq \ln G_d.
    \end{equation}
    By the definition of $K = \left\lfloor \log_2\frac{n}{d} \right\rfloor$, we have
    \[
        d 2^K \leq n < d 2^{K+1},
    \]
    and thus
    \begin{equation} \label{eq:lnn-d}
        \ln (n - d) \geq D + \ln (2^K - 1),
    \end{equation}
    and
    \begin{equation} \label{eq:lnn}
        \ln n \leq D + a(K + 1).
    \end{equation}
    By Theorem~\ref{thm:curve-bd} and \eqref{eq:lnn}, we have
    \begin{equation} \label{eq:lnDelta}
        \begin{aligned}
            \ln \Delta_u(d, n) & \leq \ln \left[ n \binom{d + K}{K} (K + T_d) \right] \\
            & = \ln n + \ln \binom{d + K}{K} + \ln (K + T_d) \\
            & \leq D + a(K+1) + \ln \binom{d + K}{K} + \ln (K + T_d).
        \end{aligned}
    \end{equation}
    Then combining \eqref{eq:lnn-d} and \eqref{eq:lnDelta} gives
    \[
        \frac{a \ln \Delta_u(d, n)}{\ln (n - d)} \leq F_d(K) \leq \sup_{K^\prime\in\mathbb{Z}_{\geq 1}} F_d(K^\prime) \leq \ln G_d,
    \]
    which proves \eqref{eq:diam-u-goal}.


    \paragraph{Case $\bm{n < 2d}$.}
    The case $n = d$ is immediate. For $d < n < 2d$, Theorem~\ref{thm:d-step} gives
    \[
        \Delta_u(d, n) \leq \Delta_u(n - d, 2(n - d)).
    \]
    Moreover, the Hirsch bound holds in dimensions at most three, with dimensions one and two being elementary and dimension three following from Theorem~\ref{thm:diam-3u}.
    
    It therefore remains to consider $4 \leq n - d < d$. By Theorem~\ref{thm:curve-bd}, we have $K = 1$ and
    \[
        \Delta_u(n - d, 2(n - d)) \leq 2(n - d) (n - d + 1) (1 + T_{n-d}).
    \]
    By Lemma~\ref{lem:lnT}, we have
    \[
        \ln T_{n-d} = \frac{(\ln(n - d))^2}{2a} + \frac{\ln(n - d)}{2} + O(1).
    \]
    Thus
    \[
        \begin{aligned}
            2(n - d) (n - d + 1) (1 + T_{n-d}) & = (n - d)^{\frac{a + \ln(n - d) + \ln(n - d + 1) + \ln(1 + T_{n-d})}{\ln(n - d)}} \\
            & = (n - d)^{\frac{\ln(n - d)}{2a} + O(1)} \\
            & = (n - d)^{\frac{1}{2} \log_2(n - d) + O(1)} \\
            & = (n - d)^{\log_2\sqrt{n - d} + O(1)} \\
            & \leq (n - d)^{\log_2\sqrt{d} + O(1)},
        \end{aligned}
    \]
    where the inequality follows from $n < 2d$. Since the $O(1)$ term is absolute and $\sqrt{d} = o(d/(\ln d)^2)$, \eqref{eq:diam-u} holds in this case for all sufficiently large $d$. Enlarging the finitely many remaining values of $G_d$, if necessary, completes the proof without affecting its asymptotic expression.
    
\end{proof}

\subsection{A Half-Logarithmic Excess-Based Bound} \label{sec:diam-d}

Combining the curved-path estimate with the quantitative $d$-step reduction in Theorem~\ref{thm:d-step}, we obtain in Theorem~\ref{thm:diam-d} an upper bound depending only on the excess $n-d$. In the regime $n-d = \Theta(d)$, this yields a leading-order improvement, halving the leading coefficient of the quasi-polynomial exponent.

\begin{proof}[Proof of Theorem~\ref{thm:diam-d}]

    Theorem~\ref{thm:d-step} gives
    \[
        \Delta_u(d, n) \leq \Delta_u(n-d, 2(n-d)).
    \]
    For $(n-d)$-dimensional pointed polyhedra with $2(n-d)$ facets, Lemma~\ref{lem:path-encoding} says that a leaf with $J$ halvings has weakly decreasing labels
    \[
        n - d \geq \kappa_1 \geq \kappa_2 \geq \cdots \geq \kappa_J \geq 4,\qquad \kappa_j 2^j \leq 2(n-d),\ \forall j.
    \]
    Dropping the ordering constraint and retaining only the coordinatewise bounds $\kappa_j\leq(n-d)2^{1-j}$, the number of length-$J$ label sequences is at most
    \[
        \prod_{j=1}^J (n-d) 2^{1 - j} = (n-d)^J 2^{-J(J-1)/2}.
    \]
    By \eqref{eq:coeff-n}, its terminal weight is at most $2(n-d)$. Summing over $J$ proves \eqref{eq:delta-path-bound-d}.

    Then by Lemma~\ref{lem:lnT},
    \[
        \begin{aligned}
            \Delta_u(n-d, 2(n-d)) & \leq 2(n-d)T_{n-d} \\
            & = (n-d)^{\frac{a + \ln(n-d) + \ln T_{n-d}}{\ln(n-d)}} \\
            & = (n-d)^{\frac{a + \ln(n-d) + \frac{(\ln(n-d))^2}{2a} + \frac{\ln(n-d)}{2} + O(1)}{\ln(n-d)}} \\
            & = (n-d)^{\frac{\ln(n-d)}{2a} + \frac{3}{2} + O(1/\ln(n-d))} \\
            & = (n-d)^{\frac{1}{2}\log_2(n-d) + \frac{3}{2} + O(1/\ln(n-d))},
        \end{aligned}
    \]
    which proves \eqref{eq:diam-d}.
    
\end{proof}

\subsection{Ordered-Block Refinements of the Excess-Based Bound} \label{sec:diam-d-order}

When estimating the number of admissible curved paths, we disregard the ordering among the $\kappa_j$'s.  In this section, we partially restore these ordering constraints, thereby improving the additive constant $3/2$ in the exponent of the excess-based bound in \eqref{eq:diam-d}.

\begin{theorem} \label{thm:diam-d-order}
    For all $n - d \geq 4$, with
    \[
        \lambda_4 = \left( \frac{181}{512} \right)^{\frac{1}{4}},\qquad c_4 = \frac{3}{2} + \log_2 \lambda_4 = 1.1249614\ldots,
    \]
    there is an absolute constant $C_4$ such that
    \begin{equation} \label{eq:delta-path-bound-d-order}
        \Delta_u(d, n) \leq C_4 (n-d) T_{\lambda_4(n-d)},
    \end{equation}
    and consequently
    \begin{equation} \label{eq:diam-d-order}
        \Delta_u(d, n) \leq (n-d)^{\frac{1}{2}\log_2(n-d) + c_4 + O\left(\frac{1}{\ln(n-d)}\right)}.
    \end{equation}
\end{theorem}
\begin{proof}

    By Theorem~\ref{thm:d-step}, it suffices to bound $\Delta_u(n-d, 2(n-d))$. We therefore apply the curved-path encoding to the reduced problem of dimension $n-d$ with $2(n-d)$ facets.
    To sharpen the resulting upper bound, we partition the positions into consecutive blocks of four and retain the ordering constraints within each block, while discarding those between different blocks.  For a block that begins at position $j$, the retained constraints are
    \[
        \kappa_j \geq \kappa_{j+1} \geq \kappa_{j+2} \geq \kappa_{j+3},
        \qquad \kappa_{j+\ell} \leq B/2^\ell,\ \ell = 0, 1, 2, 3,
    \]
    where $B = 2^{1-j}(n-d)$ is the cap at the first position of the block. Then normalize the four $\kappa$-variables by their respective upper bounds:
    \[
        x = \frac{\kappa_j}{B},
        \qquad y = \frac{\kappa_{j+1}}{B/2},
        \qquad z = \frac{\kappa_{j+2}}{B/4},
        \qquad w = \frac{\kappa_{j+3}}{B/8}.
    \]
    $(x, y, z, w) \in [0, 1]^4$ satisfies
    \begin{equation} \label{eq:p4-poly}
        y \leq 2x,
        \qquad z \leq 2y,
        \qquad w \leq 2z.
    \end{equation}
    These inequalities cut out a region of $[0, 1]^4$ with volume
    \[
        p_4 = \int_0^1 \int_0^{\min(2x,1)} \int_0^{\min(2y,1)} \int_0^{\min(2z,1)} \mathrm{d}w \mathrm{d}z \mathrm{d}y \mathrm{d}x = \frac{181}{512}.
    \]
    Moreover, the full box
    \[
        [0,B] \times [0,B/2] \times [0,B/4] \times [0,B/8]
    \]
    has volume $B^4/64$, while the subpolytope satisfying the ordering constraints has volume $p_4 B^4/64$. Since only $O(B^3)$ unit cubes meet their boundaries, each complete block contributes
    \[
        \frac{p_4 B^4/64 + O(B^3)}{B^4/64 + O(B^3)} = p_4 + O(1/B) = p_4 \left[1 + O\left(\frac{1}{B}\right)\right].
    \]
    Successive complete-block caps decrease by a factor of sixteen. Fix a sufficiently large constant $B_0$ and apply the lattice estimate only to blocks with $B_b\ge B_0$.  If $B_{\mathrm{last}}$ is the cap of the last such block, then
    \[
        \prod_{B_b\ge B_0} \left[ 1 + O\left(\frac{1}{B_b}\right) \right] 
        \leq e^{O\left( \sum_{B_b\ge B_0} \frac{1}{B_b} \right)} 
        \leq e^{O\left( \frac{1}{B_\mathrm{last}}\sum_{k=0}^{\infty} \frac{1}{16^k} \right)}
        = O(1),
    \]
    since $B_{\mathrm{last}}\geq B_0$. The remaining nonempty complete blocks are uniformly bounded in number, because their caps lie between $32$ and $B_0$, and the final incomplete block has at most three positions. Dropping their ordering conditions therefore changes the estimate only by a constant factor.
    
    Thus, writing $q = \lfloor J/4 \rfloor$ and $\lambda_4 = p_4^{1/4}$, the number of length-$J$ sequences is at most
    \[
        C_4 p_4^q (n-d)^J2^{-J(J-1)/2} = C_4\lambda_4^{-(J-4q)} \lambda_4^J(n-d)^J2^{-J(J-1)/2} 
        \leq C_4\lambda_4^{-3} \lambda_4^J(n-d)^J2^{-J(J-1)/2}.
    \]
    By \eqref{eq:coeff-n}, its terminal weight is at most $2(n-d)$. Summing over $J$ gives
    \[
        \Delta_u(d, n) \leq 2C_4\lambda_4^{-3} (n-d) T_{\lambda_4(n-d)}.
    \]
    Absorbing the fixed factor $2\lambda_4^{-3}$ into $C_4$ proves \eqref{eq:delta-path-bound-d-order}.


    Then by Lemma~\ref{lem:lnT},
    \[
        \begin{aligned}
            \Delta_u(n-d, 2(n-d)) & \leq C_4(n-d)T_{\lambda_4(n-d)} \\
            & = (n-d)^{\frac{\ln C_4 + \ln(n-d) + \ln T_{\lambda_4(n-d)}}{\ln(n-d)}} \\
            & = (n-d)^{\frac{\ln C_4 + \ln(n-d) + \frac{(\ln\lambda_4 + \ln(n-d))^2}{2a} + \frac{\ln\lambda_4 + \ln(n-d)}{2} + O(1)}{\ln(n-d)}} \\
            & = (n-d)^{\frac{\ln(n-d)}{2a} + (\frac{3}{2} + \log_2 \lambda_4) + O(1/\ln(n-d))} \\
            & = (n-d)^{\frac{1}{2}\log_2(n-d) + (\frac{3}{2} + \log_2 \lambda_4) + O(1/\ln(n-d))},
        \end{aligned}
    \]
    which proves \eqref{eq:diam-d-order}.
    
\end{proof}

For an integer $r\geq 1$, let $p_r$ denote the volume of
\[
    \left\{ (t_1,\ldots,t_r)\in[0, 1]^r : t_{j+1} \leq 2 t_j \text{ for } 1 \leq j < r \right\},
\]
and define
\[
    \lambda_r = p_r^{1/r}.
\]
The same argument used in the proof of Theorem~\ref{thm:diam-d-order} applies to any fixed block length $r$, with the boundary and incomplete-block contributions absorbed into a constant depending only on $r$. This yields the following corollary.
\begin{corollary}
    For every fixed $r\geq 1$, there exists a constant $C_r$ such that, for all $n-d\geq 4$,
    \begin{equation} \label{eq:delta-path-bound-d-order-r}
        \Delta_u(d, n) \leq C_r (n - d) T_{\lambda_r(n-d)}.
    \end{equation}
    Consequently, uniformly in $d$, as $n-d\to\infty$,
    \begin{equation} \label{eq:diam-d-order-r}
        \Delta_u(d, n) \leq (n - d)^{\frac{1}{2}\log_2(n - d) + c_r + O_r\left(\frac{1}{\ln(n-d)}\right)},
        \qquad c_r = \frac{3}{2} + \log_2 \lambda_r,
    \end{equation}
    where the constant implicit in $O_r(\cdot)$ may depend on $r$ but not on $d, n$.
\end{corollary}

Deleting the single constraint between positions $r$ and $r + 1$ gives $p_{r + s} \leq p_r p_s$. Thus, the sequence $\{p_r\}$ is submultiplicative, and Fekete's lemma yields
\[
    \rho = \lim_{r\to\infty} p_r^{1/r} = \inf_{r\geq 1} p_r^{1/r}.
\]

Numerically, we obtain
\[
    \rho \approx 0.6720,\qquad c_* = \frac{3}{2} + \log_2 \rho \approx 0.9265.
\]
Choosing increasingly large fixed blocks makes $c_r$ arbitrarily close to $c_*$. Thus, $c_*$ is the limiting value of the additive constants obtained by this block-ordering refinement.

\subsection{A Target-Dimension Refinement of the General Bounds} \label{sec:diam-d-dim}

Section~\ref{sec:diam-d} shows that, in the regime $n - d = \Theta(d)$, the general excess-based bound \eqref{eq:diam-d} halves the leading coefficient of $\log_2 d$ in the exponent of \eqref{eq:diam-u} from $1$ to $1/2$.  We next consider a regime with moderately more facets, namely $n = d^{1/\gamma+o(1)}$ for a fixed $0 < \gamma < 1$.

\begin{theorem} \label{thm:off-diagonal}
    Fix $0 < \gamma < 1$, and suppose that $n = d^{1/\gamma + o(1)}$ as $d\to\infty$. Then,
    \begin{equation} \label{eq:off-diagonal-d}
        \Delta_u(d, n)
        \leq (n - d)^{ \left( 1 - \frac{\gamma}{2} + o(1) \right)\log_2 d }.
    \end{equation}
    Equivalently,
    \begin{equation} \label{eq:off-diagonal-excess}
        \Delta_u(d, n)
        \leq (n - d)^{ \left( \frac{1}{2} - \frac{(1-\gamma)^2}{2} + o(1) \right)\log_2(n-d) }.
    \end{equation}
\end{theorem}

\begin{proof}
    Since $1/\gamma > 1$, we have $d = o(n)$, and $n \geq 2d$ for all sufficiently large $d$. Moreover,
    \[
        n - d = n(1 - o(1)) = d^{1/\gamma + o(1)},
    \]
    and thus,
    \begin{equation} \label{eq:gamma-log-relation}
        D = \left(\gamma+o(1)\right)\ln (n - d).
    \end{equation}
    
    Since $\ln n = \ln (n - d)+o(1)$, \eqref{eq:gamma-log-relation} gives
    \begin{equation} \label{eq:gamma-K}
        K = \left\lfloor \log_2\frac{n}{d} \right\rfloor
        = \left( \frac{1 - \gamma}{a\gamma} + o(1) \right) D.
    \end{equation}
    Moreover, $K=\Theta(D)=o(d)$. Therefore, Stirling's formula yields
    \[
    \begin{aligned}
        \ln\binom{d+K}{K} & = K\left(D-\ln K+1\right) + O\left( \ln(K+1) + \frac{K^2}{d} \right) \\
        & = \left( \frac{1-\gamma}{a\gamma} + o(1) \right) D^2.
    \end{aligned}
    \]
    
    Meanwhile, Lemma~\ref{lem:lnT} and \eqref{eq:gamma-log-relation} give
    \[
        \ln T_d = \frac{D^2}{2a} + \frac{D}{2} + O(1).
    \]
    Since $\gamma > 0$ is fixed, $T_d$ grows faster than $K$, and hence
    \[
        \ln(K + T_d) = \ln T_d + o(1).
    \]
    
    Applying Theorem~\ref{thm:curve-bd}, we obtain
    \[
        \ln\Delta_u(d,n) \leq \ln n + \ln\binom{d+K}{K} + \ln(K+T_d) = \frac{1}{a} \left( \frac{1}{\gamma} - \frac{1}{2} + o(1) \right) D^2.
    \]
    Exponentiating gives
    \[
        \Delta_u(d,n) \leq (n-d)^{\left( 1 - \frac{\gamma}{2} + o(1) \right)\log_2 d},
    \]
    or equivalently,
    \[
        \Delta_u(d,n) \leq (n-d)^{ \left(\gamma - \frac{\gamma^2}{2} + o(1) \right)\log_2(n-d) }
        = (n-d)^{ \left(\frac{1}{2} - \frac{(1-\gamma)^2}{2} + o(1) \right)\log_2(n-d) }.
    \]
\end{proof}

Note that in the regime $n = d^{1/\gamma + o(1)}$ with fixed $0 < \gamma < 1$, our target-dimension refinement reduces the coefficient of $\log_2 d$ in \eqref{eq:diam-u} from $1$ to $1 - \gamma/2$; equivalently, for the excess-based bound, it improves the coefficient of $\log_2(n-d)$ in \eqref{eq:diam-d} from $1/2$ to $1/2 - (1-\gamma)^2/2$.


\subsection{Transition to the Tail Regime} \label{sec:diam-tail}

In this section, we consider the regime in which the number of facets is particularly large relative to the dimension, specifically $(\ln n)/d \to \infty$. We show that, in this regime, the finite curved-path estimate yields a tail-almost-linear bound.

\begin{corollary} \label{cor:diam-tail}    
    Along every sequence of integer pairs $(d,n)$ satisfying $d\geq 4$, $n \geq 2d$, and $(\ln n) / d \to \infty$, we have
    \begin{equation} \label{eq:diam-tail}
        \Delta_u (d, n) \leq n^{1 + o(1)}.
    \end{equation}
\end{corollary}

\begin{proof}
    By Theorem~\ref{thm:curve-bd},
    \[
        \Delta_u(d, n) \leq n \binom{d + K}{K} (K + T_d).
    \]
    Using
    \[
        K \leq \frac{\ln n}{a},
    \]
    and \eqref{eq:binom-dK-K-bound}
    \[
        \binom{d + K}{K} \leq \left( e \left( 1 + \frac{K}{d}\right) \right)^d,
    \]
    we have
    \[
        \begin{aligned}
            \ln \Delta_u(d, n) & \leq \ln n + d \ln \left( e \left( 1 + \frac{K}{d}\right) \right) + \ln (K + T_d) \\
            & \leq \ln n + d \ln \left( e \left( 1 + \frac{\ln n}{ad}\right) \right) + a + \max\{\ln K, \ln T_d\} \\
            & = \ln n + d \ln \left( e \left( 1 + \frac{\ln n}{ad}\right) \right) + O(\ln \ln n + D^2) \\
            & = \ln n + o(\ln n).
        \end{aligned}
    \]
    Therefore,
    \[
        \Delta_u(d, n) \leq n^{1 + o(1)}.
    \]

    
\end{proof}

The same calculation also yields a polynomial tail bound in $n$ when $(\ln n)/d=\Theta(1)$. We include Corollary~\ref{cor:diam-tail} and this observation only for completeness, in order to describe the transition of the finite curved-path estimate from the quasi-polynomial regime to tail-almost-linear behavior. These conclusions do not improve the best known tail bounds, since the generalized Larman bound gives quantitatively stronger estimates throughout this range.

\section{Sharpness within Positive Curved-Path Counting} \label{sec:sharp-u}


We now ask whether a more precise count of the recursion paths generated by \eqref{eq:GK} could reduce the leading constant $4\ln2$ in \eqref{eq:diam-u}. After fully expanding the recurrence and discarding the negative internal terms, we exhibit a family of admissible leaves whose total positive contribution already has the same leading order.

\begin{theorem} \label{thm:sharp-u}
    Let
    \[
        n = d2^{\left\lfloor \frac{D^2}{2a} \right\rfloor}.
    \]
    As $d \to \infty$, there exists a family of admissible recursion leaves whose total contribution is at least
    \begin{equation} \label{eq:sharp-major}
        (n-d)^{ \log_2 \left( (4\ln2 + o(1))\frac{d}{(\ln d)^2} \right) }.
    \end{equation}
    Consequently, the leading constant $4\ln 2$ in the general curved-path bound \eqref{eq:diam-u} cannot be improved by a more accurate positive count of the same recursion leaves.
\end{theorem}
\begin{proof}

    It suffices to construct the asserted family of admissible leaves. We consider
    \[
        K = \left\lfloor \log_2\frac{n}{d} \right\rfloor = \left\lfloor \frac{D^2}{2a} \right\rfloor,
        \qquad L = \left\lfloor \log_2 \frac{d}{16D} \right\rfloor,
    \]
    so that $n = d 2^K$. After the recursion is completed down to dimension three, each path in this family has accumulated exactly $K+L$ halvings.

    For the first $K$ positions, restrict every dimension label to satisfy
    \[
        d \geq \kappa_1 \geq \kappa_2 \geq \cdots \geq \kappa_K \geq \left\lceil \frac{d}{D} \right\rceil.
    \]
    Since $n = d2^K$ and $\kappa_j \leq d$, the constraint $\kappa_j 2^j \leq n$ holds for these first $K$ labels. The interval $[\lceil d / D \rceil, d]$ contains
    \[
        d - \lceil d/D\rceil + 1 = \left\lfloor \left(1 - \frac{1}{D}\right) d \right\rfloor + 1
    \]
    integers. Hence, the number of weakly decreasing sequences of length $K$ is
    \[
        \binom{\lfloor (1 - 1/D) d \rfloor + K}{K}.
    \]

    For the next $L$ positions, require the label at position $K + \ell$ to satisfy
    \begin{equation} \label{eq:kappa-L-bound}
        \left\lfloor \frac{d}{2^{\ell-1} D} \right\rfloor
        \geq \kappa_{K+\ell}
        \geq \left\lceil \frac{d}{2^\ell D} \right\rceil,
        \qquad 1 \leq \ell \leq L.
    \end{equation}
    These adjacent intervals preserve the weak ordering. Moreover, for sufficiently large $d$, we have $D \geq 2$ and then $\kappa_{K+\ell} 2^{K+\ell} \leq n$. Thus, the resulting sequences are admissible. Each interval $\left[ \lceil d / (2^\ell D) \rceil, \lfloor d/(2^{\ell-1}D) \rfloor \right]$ contains $\Theta(d/(2^\ell D))$ integers, with an implicit constant uniform in $\ell$.

    Therefore, the construction gives at least
    \begin{equation} \label{eq:sharp-path-count}
        \binom{\lfloor (1 - 1/D) d \rfloor + K}{K} \prod_{\ell=1}^L \Theta \left( \frac{d}{2^\ell D} \right)
    \end{equation}
    admissible paths. All these paths have the same terminal contribution. For pointed polyhedra, Theorem~\ref{thm:diam-3u} gives
    \[
        2^{K+L} \Delta_u\left(3, \left\lfloor\frac{n}{2^{K+L}}\right\rfloor\right)
        = 2^{K+L} \left(\left\lfloor\frac{n}{2^{K+L}}\right\rfloor - 3\right)
        = n\left( 1 + O\left(\frac{1}{D}\right) \right).
    \]
    For polytopes, Theorem~\ref{thm:diam-3b} gives
    \[
        2^{K+L} \Delta_b\left(3, \left\lfloor\frac{n}{2^{K+L}}\right\rfloor\right)
        = \frac{2}{3}n\left( 1 + O\left(\frac{1}{D}\right) \right).
    \]
    Thus, in either case, the logarithm of the terminal contribution is
    \begin{equation} \label{eq:sharp-log-weight}
        \ln n + O(1).
    \end{equation}

    It remains to estimate the two factors arising from the path count \eqref{eq:sharp-path-count}.
    Since
    \[
        \ln \lfloor (1 - 1/D) d \rfloor
        = D + \ln \left(1 - \frac{1}{D}\right) + O\left(\frac{1}{d}\right)
        = D - \frac{1}{D} + O\left(\frac{1}{D^2}\right),
    \]
    and
    \[
        \ln K = 2 \ln D - \ln(2a) + o(1),
    \]
    from \eqref{eq:binom}, we have
    \begin{equation} \label{eq:sharp-K}
        \begin{aligned}
            \ln \binom{\lfloor (1 - 1/D) d \rfloor + K}{K} & = K\left( \ln \lfloor (1 - 1/D) d \rfloor - \ln K + 1 \right) + O\left( \ln(K+1) + \frac{K^2}{\lfloor (1 - 1/D) d \rfloor}\right) \\
            & = K\left( D - 2 \ln D + \ln(2a) + 1 \right) - \frac{K}{D} + O\left( \frac{K}{D^2} + \ln(K+1) + \frac{K^2}{d}\right) \\
            & = K\left( D - 2 \ln D + \ln(2a) + 1 \right) + O(D).
        \end{aligned}
    \end{equation}

    Let $N_\ell$ denote the number of integers in the interval specified by \eqref{eq:kappa-L-bound}. Since $L = \frac{D - \ln D}{a} + O(1)$, and the implicit constants in $N_\ell = \Theta(d / (2^\ell D))$ are uniform in $\ell$, taking logarithms gives
    \begin{equation} \label{eq:sharp-L}
        \begin{aligned}
             \ln \prod_{\ell=1}^L N_\ell 
            & = \sum_{\ell=1}^L [D - a\ell - \ln D + O(1)] \\
            & = L(D - \ln D) - \frac{aL(L+1)}{2} + O(L) \\
            & = \frac{(D - \ln D)^2}{2a} + O(D) \\
            & = \frac{D^2}{2a} + o(D^2).
        \end{aligned}
    \end{equation}

    The selected family of paths contributes
    \[
        \Theta(n) \binom{\lfloor (1 - 1/D) d \rfloor + K}{K} \prod_{\ell=1}^L \Theta \left( \frac{d}{2^\ell D} \right).
    \]
    To compare this contribution with the bound in \eqref{eq:diam-u}, we use the logarithmic normalization introduced in \eqref{eq:diam-u-goal}. Hence, by \eqref{eq:sharp-log-weight}, \eqref{eq:sharp-K}, and \eqref{eq:sharp-L}, the normalized logarithm of the total contribution is bounded below by
    \[
    \begin{aligned}
        & \frac{a \left[ \ln n + O(1) + K\left( D - 2 \ln D + \ln(2a) + 1 \right) + O(D) + \frac{D^2}{2a} + o(D^2) \right] }{\ln(d2^K - d)} \\
        =\ & \frac{a \left[ D + aK + K\left( D - 2 \ln D + \ln(2a) + 1 \right) + \frac{D^2}{2a} + o(D^2) \right] }{D + aK + O(1)} \\
        =\ & D - 2\ln D + \ln(4\ln 2) + o(1) \\
        =\ & \ln \left( (4\ln 2 + o(1))\frac{d}{(\ln d)^2} \right).
    \end{aligned}
    \]
    Reversing the normalization in \eqref{eq:diam-u-goal} gives the lower bound in \eqref{eq:sharp-major}. Together with the positive path-counting upper bound established in Theorem~\ref{thm:diam-u}, this proves the claimed sharpness.

\end{proof}

We emphasize that Theorem~\ref{thm:sharp-u} is a sharpness statement only for the positive path-counting argument; the terminal contributions need not be simultaneously realized by any polyhedron. It neither constructs polyhedra with diameters of this order nor provides a lower bound on the true polyhedral diameter.






\section{Conclusion} \label{sec:conclu}

In this paper, we obtain asymptotically improved upper bounds for the diameters of pointed polyhedra through a curved-path refinement. Blockwise ordering improves the additive constant in the exponent, while a target-dimension refinement improves the leading coefficient in certain asymptotic regimes. The same curved-path estimate also recovers almost-linear behavior in the deep-tail regime.

Nevertheless, our general bounds remain quasi-polynomial and therefore do not establish the polynomial Hirsch conjecture. Our sharpness analysis shows that the leading term is sharp for the positive path-counting argument used here.


\section*{Acknowledgments}
This research is partially supported by NSFC [Grant NSFC-72225009, 72688301, 72394360, 72394365].

AI tools assisted with parts of the theoretical development; the authors verified all results and take full responsibility for the paper.


\bibliographystyle{plainnat}
\bibliography{ref}


\end{document}